\documentclass[12pt]{elsarticle}

\makeatletter
\def\ps@pprintTitle{%
  \let\@oddhead\@empty
  \let\@evenhead\@empty
  \let\@oddfoot\@empty
  \let\@evenfoot\@oddfoot
}
\makeatother

\usepackage{amssymb}

\newtheorem{theorem}{Theorem}[section]
\newtheorem{lemma}[theorem]{Lemma}

\newtheorem{proposition}[theorem]{Proposition}
\newdefinition{definition}{Definition}
\newdefinition{remark}{Remark}
\newdefinition{example}{Example}
\newdefinition{question}{Question}
\newproof{proof}{Proof}

\begin{document}

\begin{frontmatter}



\title{Baire and weakly Namioka spaces}


\author{Zbigniew Piotrowski}

\address{Department of Mathematics, Youngstown State University, Youngstown, OH 44555, USA}

\author{Russell Waller}

\address{Department of Mathematics, Florida State University, Tallahassee, FL 32306, USA}

\begin{abstract}
Recall that a Hausdorff space $X$ is said to be \emph{Namioka} if for every compact (Hausdorff) space $Y$ and every metric space $Z$, every separately continuous function $f:X\times{Y}\rightarrow{Z}$ is continuous on $D\times{Y}$ for some dense $G_\delta$ subset $D$ of $X$. It is well known that in the class of all metrizable spaces, Namioka and Baire spaces coincide \cite{Saint-Raymond83}. Further it is known that every completely regular Namioka space is Baire and that every separable Baire space is Namioka \cite{Saint-Raymond83}.

In our paper we study spaces $X$, we call them \emph{weakly Namioka}, for which the conclusion of the theorem for Namioka spaces holds provided that the assumption of compactness of $Y$ is replaced by \emph{second countability} of $Y$. We will prove that in the class of all completely regular separable spaces and in the class of all perfectly normal spaces, $X$ is Baire if and only if it is weakly Namioka.
\end{abstract}


\end{frontmatter}


\section{Introduction}
\label{}

All the spaces considered in this paper are assumed to be Hausdorff. For spaces $X$, $Y$, and $Z$, we say a function $f:X\times{Y}\rightarrow{Z}$ is \emph{continuous with respect to $x$} if $f|_{X\times{\{y\}}}$ is continuous for every $y\in{Y}$. Similarly, $f$ is \emph{continuous with respect to $y$} if $f|_{\{x\}\times{Y}}$ is continuous for every $x\in{X}$. We say $f$ is \emph{separately continuous} if $f$ is continuous with respect to $x$ and continuous with respect to $y$.  In \cite{Saint-Raymond83}, J. Saint-Raymond shows that every completely regular Namioka space is Baire and that every metrizable or separable Baire space is Namioka, thus providing a characterization of Baire spaces in terms of Namioka spaces. In this paper, we will be concerned with finding a similar characterization of Baire spaces using weakly Namioka spaces. 


\section{Main Results}
\label{}

Let us start with the following:

\begin{definition}
A space $X$ is \emph{weakly Namioka} if for every second countable space $Y$ and every metric space $Z$, every separately continuous function $f:X\times{Y}\rightarrow{Z}$ is continuous on $D\times{Y}$ for some dense $G_\delta$ subset $D$ of $X$.
\end{definition}

In \cite{Calbrix}, J. Calbrix and J.P. Troallic show that given a sequence of open subsets $(U_n)_{n\in{\mathbb{N}}}$ of $Y$ and a metric space $M$, there is a residual set $R$ in $X$ such that the separately continuous function $f:X\times{Y}\rightarrow{M}$ is continuous at each point of $R\times{Q}$, where $Q$ is the set of points $y\in{Y}$ admitting a subsequence of $(U_n)_{n\in{\mathbb{N}}}$ as a neighborhood basis. In \cite {BT}, \cite{M}, \cite{MR} \cite{Mas}, \cite{MasNest}, \cite{MN}, \cite{Nest}, it is shown that for a topological space $X$, a second countable space $Y$, a metric space $M$, and $f:X\times{Y}\rightarrow{M}$ such that $f$ is continuous with respect to $y$ and $f|_{X\times{\{y\}}}$ is continuous for every $y\in{D}$ for some $D$ dense in $Y$, there is a residual set $A$ in $X$ such that $f$ is continuous at each point of $A\times{Y}$. Here we offer a new proof of this result:

\begin{lemma}\label{BrownLemma}
Let $Y$ be second countable, $(M,d)$ be metric and $f:X\times{Y}\rightarrow{M}$ be such that $f$ is continuous with respect to $y$ and $f|_{X\times{\{y\}}}$ is continuous for every $y\in{D}$ for some $D$ dense in $Y$. There is then a residual set $A$ in $X$ such that $f$ is continuous on $A\times{Y}$.
\end{lemma}

\begin{proof}
Without loss of generality, assume $X$ is of second category, for otherwise there is nothing to prove. In fact, if $X$ is of first category then the set $A$ can possibly be a priori empty. Assume $M_1=\{x\in{X}: \exists{y\in{Y}}$, $f$ is discontinuous at $(x,y)\}$ is of second category in $X$.

For each $x\in{M_1}$, let $y_x$ be an element of $Y$ such that $f$ is discontinuous at $(x,y_x)$. Let $\epsilon>0$ be such that the set $M_2=\{x\in{M_1} :$ for any open neighborhood $O$ of $(x,y_x)$ in $X\times{Y}$, there exists $(u,v)\in{O}$ such that $d(f(x,y_x), f(u,v))\geq{\epsilon}\}$ is of second category in $X$. 

Let $B_1,B_2,B_3, ...$ be a countable base for $Y$.
Now, let $n$ be an index of $B_n$ such that the set

$(*)$ $M_3=\{x\in{M_2}: y_x\in{B_n}$ and $\forall{y\in{B_n}},d(f(x,y_x),f(x,y))<\frac{\epsilon}{6}\}$ 

\noindent is of second category in $X$. Let $y_1\in{D\cap{B_n}}$ and let $U$ be an open set of $X$ such that $U\cap{M_3}$ is dense in $U$. Let $x_0\in{U\cap{M_3}}$ and let $V$ be an open subset of $U$ containing $x_0$ such that for every $x\in{V}$, 

$(**)$ $d(f(x_0,y_1),f(x,y_1))<\frac{\epsilon}{6}$.

\noindent There is a point $(a,b)\in{V\times{B_n}}$, such that

$(***)$ $d(f(a,b), f(x_0,y_{x_0}))\geq{\epsilon}$. 

\noindent Let $G$ be an open subset of $B_n$, containing $b$ such that for every $y\in{G}$,

$(*v)$ $d(f(a,b), f(a,y))<\frac{\epsilon}{6}$. 

\noindent Let $c\in{G\cap{D}}$. Finally, let $W$ be an open subset of $V$ containing $a$ such that for every $x\in{W}$,

$(v)$ $d(f(a,c)), f(x,c))<\frac{\epsilon}{6}$.

\noindent Let $x_1\in{W\cap{M_3}}$. Now:

$d(f(a,b),f(a,c))<\frac{\epsilon}{6}$, by $(*v)$

$d(f(a,c), f(x_1,c))<\frac{\epsilon}{6}$, by $(v)$

$d(f(x_1,c), f(x_1,y_{x_1}))<\frac{\epsilon}{6}$ by $(*)$

$d(f(x_1,y_{x_1}), f(x_1,y_1))<\frac{\epsilon}{6}$ by $(*)$

$d(f(x_1,y_1), f(x_0,y_1))<\frac{\epsilon}{6}$, by $(**)$

$d(f(x_0,y_1), f(x_0,y_{x_0}))<\frac{\epsilon}{6}$, by $(*)$.

\noindent Hence: $d(f(a,b),f(x_0,y_{x_0}))<\epsilon$, contradicting $(***)$.
\qed
\end{proof}

From either of the above mentioned results comes the following theorem as an easy corollary:

\begin{theorem}
\label {Brown}
Baire spaces are weakly Namioka.
\end{theorem}

All that remains, then, in finding a characterization of Baire spaces using weakly Namioka spaces is to determine under what conditions the converse is true. We turn to this task now.
The proof of the forthcoming theorem is an adaptation of the proof of Theorem 3 of \cite{Saint-Raymond83}.

\begin{theorem}
\label{big}
Completely regular, separable, weakly Namioka spaces are Baire.
\end{theorem}
\begin{proof}

Let $X$ be a completely regular, separable space and assume that $X$ is not Baire. Then there exists a nonempty open set $K$ in $X$ such that $K\subset{\bigcup_{n\in{\gamma}}K_n}$, where $\gamma$ is a countable indexing set and $K_n$ is closed and nowhere dense for each $n\in{\gamma}$.
Let $S$ be a countable dense subset of $X$.
Then for any $n\in{\gamma}$, $S\setminus{K_n}$ is dense in $X$. For each $n\in{\gamma}$ choose an indexing set $\gamma_n$ such that $S\setminus{K_n}=\{s_{n,i}:i\in{\gamma_n}\}$ and write $S_n=S\setminus{K_n}=\{s_{n,i}:i\in{\gamma_n}\}$.

Choose $n\in{\gamma}$ and let $\Phi$ be a set of continuous functions $\varphi:X\rightarrow{[0,1]}$ such that $\varphi(K_n)\subset{\{0\}}$, $\varphi(s_{n,i})=1$ for some $s_{n,i}\in{S_n}$, and such that $\varphi(x)\neq{0}\Rightarrow{\varphi'(x)=0}$ for all $x\in{X}$ and any $\varphi,\varphi'\in{\Phi}$, $\varphi\neq{\varphi'}$ (note that while for every $\varphi\in{\Phi}$ there is a corresponding $s_{n,i}\in{S_n}$ such that $\varphi(s_{n,i})=1$, the reverse is not true).
Then for each $\varphi\in{\Phi}$ the set $A_\varphi=\{x:\varphi(x)\neq{0}\}$ is nonempty and  $A_\varphi\cap{A_{\varphi'}}=\emptyset$ if $\varphi\neq{\varphi'}$ (note that supp($\varphi$) is the closure of $A_\varphi$).
Let $P_\Phi=\{A_\varphi:\varphi\in{\Phi}\}$ be the set of all $A_\varphi$'s for $\Phi$.
Let $F_n$ be the set of all $\Phi$'s so defined (for fixed $n\in{\gamma}$), and take the partial order $\prec$ on $F_n$ to be $\Phi\prec{\Phi'}$ if $P_\Phi\subset{P_{\Phi'}}$.
Take a simply ordered subset $H$ of $F_n$, and let $H^*=\bigcup_{\Phi\in{H}}\Phi$.
Then for any $\varphi\in{H^*}$ there exists $s_{n,i}\in{S_n}$ such that $\varphi(K_n)\subset{\{0\}}$ and $\varphi(s_{n,i})=1$.

Seeking contradiction, assume $A_{\varphi}\cap{A_{\varphi'}}\neq{\emptyset}$ for some $\varphi,\varphi'\in{H^*}$, $\varphi\neq{\varphi'}$.
There exists $\Phi,\Phi'\in{H}$ such that $\varphi\in{\Phi}$ and $\varphi'\in{\Phi'}$.
Since $H$ is simply ordered, it follows that $\Phi\prec{\Phi'}$ or $\Phi'\prec{\Phi}$. Without loss of generality, assume $\Phi\prec{\Phi'}$.
Then $P_\Phi\subset{P_{\Phi'}}$, so $A_\varphi=A_{\varphi''}$ for some $\varphi''\in{\Phi'}$.
But then $A_{\varphi''}\cap{A_{\varphi'}}\neq{\emptyset}$, a contradiction.
So for all $x\in{X}$, $\varphi(x)\neq{0}\Rightarrow{\varphi'(x)=0}$ if $\varphi\neq{\varphi'}$.
Thus $H^*\in{F_n}$, so $H^*$ is an upper bound of $H$ in $F_n$.
Zorn's Lemma therefore guarantees the existence of a maximal element of $F_n$.
For each $n$, then, take $\Phi_n$ to be the maximal element of $F_n$.
Note that each $\Phi_n$ is countable since it must be no larger than $S_n$.

Seeking contradiction, assume now that for some $n\in{\gamma}$, $\{x\in{X}:\exists{\varphi\in{\Phi_n}},\varphi(x)\neq{0}\}$ is not dense in $X$. Then there exists an open set $U$ in $X$ such that for all $\varphi\in{\Phi_n}$,  $\varphi(U)\subset{\{0\}}$.
Since $S_n$ is dense in $X$, there exists $s_{n,u}\in{S_n}$ such that $s_{n,u}\in{U}$. Then since $X$ is completely regular, there exists continuous $g:X\rightarrow{[0,1]}$ such that $g(s_{n,u})=1$ and $g((X\setminus{U})\cup{K_n})\subset{\{0\}}$.
But $\Phi_n\prec{\Phi_n\cup{\{g\}}}$, contradicting the maximality of $\Phi_n$. So for all $n\in{\gamma}$, $\{x\in{X}:\exists{\varphi\in{\Phi_n}}, \varphi(x)\neq{0}\}$ is dense in $X$.

Let $\Phi_n$ be given the discrete topology. Since $\Phi_n$ and $[0,1]$ are both locally compact and Hausdorff, their product $\Phi_n\times{[0,1]}$ is locally compact and Hausdorff, and so admits a one-point compactification. For each $n\in{\gamma}$, let $Y_n$ be the one-point compactification of $\Phi_n\times{[0,1]}$ where $\lambda_n$ is the point at infinity. 
Let $Y=\coprod_{n\in{\gamma}}{Y_n}$ be the disjoint union equipped with the coherent topology. Define $f:X\times{Y}\rightarrow{[0,1]}$ by

$f(x,y) = 
\left\{
\begin{array}{lr}
f(x,\varphi,t)=\frac{(2t)(\varphi(x))}{t^2+(\varphi(x))^2}& t\neq{0}$ and $\forall{n\in{\gamma}},  y\neq{\lambda_n} \\
0&$otherwise.$
\end{array}
\right.$

To show the separate continuity of $f$, first fix $y_0\in{Y}$.
If $y_0=(\varphi,0)$ or $y_0=\lambda_n$ for some $n\in{\gamma}$, then $f(x,y_0)=0$ for all $x\in{X}$. So $f|_{X\times{\{y_0\}}}\subset{\{0\}}$ and is thus continuous.
If $y_0\neq{(\varphi,0)}$ and $y_0\neq{\lambda_n}$ for all $n\in{\gamma}$, then $f(x,y_0)=f(x,\varphi,t)=\frac{(2t)\varphi(x))}{t^2+(\varphi(x))^2}$ for some fixed $\varphi$ and $t$,
so $f|_{X\times{\{y_0\}}}$ is continuous by the continuity of $\varphi$. 
Therefore $f|_{X\times{\{y\}}}$ is continuous for any fixed $y\in{Y}$.

Now fix $x_0\in{X}$.
Take an open set $(a,b)$ in $[0,1]$ (or take $(a,1]$ without loss of generality).
For each $n\in{\gamma}$, $f^{-1}(a,b)\cap{\{x_0\}\times{Y_n}}=\{(x_0,\varphi,t):\frac{(2t)(\varphi(x_0))}{t^2+(\varphi(x_0))^2}\in{(a,b)}\}$ for a particular $\varphi\in{\Phi_n}$ (since $\varphi(x)\neq{0}\Rightarrow{\varphi'(x)\neq{0}}$ if $\varphi\neq{\varphi'}$).
Thus $f^{-1}(a,b)\cap{\{x_0\}\times{Y_n}}$ is open in $\{x_0\}\times{Y_n}$ since $\frac{(2t)(\varphi(x_0))}{t^2+(\varphi(x_0))^2}$ is continuous as a function of $t$ ($t\neq{0}$ since $0\not\in{(a,b)}$).
So $f^{-1}(a,b)\cap{\{x_0\}\times{Y}}$ is open in $\{x_0\}\times{Y}$.
Now take an open set $[0,a)$ in $[0,1]$.
For each $n\in{\gamma}$, $f^{-1}[0,a)$ contains $(x_0,\lambda_n)$ and there is at most one $\varphi$ in $\Phi_n$ such that for some $t$, $(x_0,\varphi,t)\not\in{f^{-1}[0,a)}$. If there is no such $\varphi$ in $\Phi_n$, then $f^{-1}[0,a)\cap{\{x_0\}\times{Y_n}}=\emptyset$. If such a $\varphi$ does exist in $\Phi_n$, then
by the continuity of $\frac{(2t)(\varphi(x_0))}{t^2+(\varphi(x_0))^2}$ in terms of $t$ (for fixed $x_0$ and $\varphi$), it follows that for our particular $\varphi\in{\Phi_n}$ the set $\{t\in{[0,1]}:(x_0,\varphi,t)\not\in\ f^{-1}[0,a)\}$ is closed in $[0,1]$, and so is compact.
Thus $(\{x_0\}\times{Y_n})\setminus{f^{-1}[0,a)}$ is closed and compact when restricted to $\{x_0\}\times{(\Phi_n\times{[0,1]})}$, and $(x_0,\lambda_n)\in{f^{-1}[0,a)}$ for each $n\in{\gamma}$. So $f^{-1}[0,a)\cap{\{x_0\}\times{Y_n}}$ is open in $\{x_0\}\times{Y_n}$ for each $n\in{\gamma}$, and therefore $f^{-1}[0,a)\cap{\{x_0\}\times{Y}}$ is open in $\{x_0\}\times{Y}$.
It follows that $f|_{\{x\}\times{Y}}$ is continuous for any fixed $x\in{X}$, and the separate continuity of $f$ is established.

We now demonstrate that  for any dense $G_\delta$ subset $D$ of $X$, $f$ is not continuous on $D\times{Y}$. To do so it suffices to show that for each $x\in{K}$, there exists $y\in{Y}$ such that $f$ is discontinuous at $(x,y)$.
Choose $x_K\in{K}$.
Then for some $n\in{\gamma}$, $x_K\in{K_n}$.
Since $\{x\in{X}: \exists{\varphi\in{\Phi_n}}, \varphi(x)\neq{0}\}$ is dense in $X$, there is a directed set $A$ and a net $(x_\alpha)_{\alpha\in{A}}$ in $\{x\in{X}: \exists{\varphi\in{\Phi_n}}, \varphi(x)\neq{0}\}$ such that  $x_\alpha\rightarrow{x_K}$.
Since each $x_\alpha$ is in $\{x\in{X}: \exists{\varphi\in{\Phi_n}}, \varphi(x)\neq{0}\}$ there exists $\varphi'\in{F_n}$ and $r\in{(0,1]}$ such that $\varphi'(x_\alpha)=r$.
                Let $y_\alpha=(\varphi',r)$ for each $\alpha\in{A}$ and let $p_\alpha=(x_\alpha,y_\alpha)$. Then $(y_\alpha)_{\alpha\in{A}}$ is a net in $Y_n$ such that for each $\alpha\in{A}$, $f(p_\alpha)=f(x_\alpha,y_\alpha)=f(x_\alpha,\varphi',r) = \frac{(2r)(\varphi'(x_\alpha))}{r^2+(\varphi'(x_\alpha))^2} = \frac{(2r)(r)}{r^2+r^2} = 1$.
           Since $(y_\alpha)$ is a net contained in the compact subspace $Y_n$ of $Y$, there is a subnet $(\tilde{y}_{\beta})_{\beta\in{B}}$ of $(y_\alpha)$ that converges to some point $y\in{Y_n}$.
So then $(\tilde{x}_{\beta})_{\beta\in{B}}$ is a subnet of $(x_\alpha)_{\alpha\in{A}}$ and thus $\tilde{x}_{\beta}\rightarrow{x_K}$.
                   Let $\tilde{p}_\beta=(\tilde{x}_\beta,\tilde{y}_\beta)$.
Since $(\tilde{x}_{\beta})\rightarrow{x_K}$ in $X$ and  $(\tilde{y}_{\beta})\rightarrow{y}$ in $Y$, $\tilde{p}_\beta\rightarrow{(x_K,y)}$ in $X\times{Y}$.
But since $x_K\in{K_n}$ (and $y\in{Y_n}$) and $(\tilde{p}_\beta)_{\beta\in{B}}$ is a subnet of $(p_\alpha)_{\alpha\in{A}}$, we have

\noindent
$f(x_K,y) = 
\left\{
\begin{array}{lr}
f(x_K,\varphi,t)=\frac{(2t)(\varphi(x_K))}{t^2+(\varphi(x_K))^2}=\frac{0}{t^2}=0&t\neq{0}$ and $\forall{n\in{\gamma}}, y\neq{\lambda_n}\\
0&$otherwise,$
\end{array}
\right.$

\noindent and $f(\tilde{p}_\beta)=1$ for every $\beta\in{B}$. 
So $\tilde{p}_\beta\rightarrow{(x_K,y)}$ in $X\times{Y}$ but $f(\tilde{p}_\beta)\not\rightarrow{f(x_K,y)}$ in $[0,1]$.
Thus $f$ is discontinuous at $(x_K,y)$, and it follows that $X$ is not weakly Namioka.
\qed
\end{proof}

\begin{theorem}
\label{small}
Perfectly normal weakly Namioka spaces are Baire.
\end{theorem}

\begin{proof}
Let $X$ be a perfectly normal space and assume that $X$ is not a Baire space. There then exists an open set $U$ in $X$ that is of first category and of type $F_\sigma$. Let $Y$ be a second countable completely regular space with a non-isolated point $y_0$. By Theorem 5, p. 1111 of \cite{Maslyuchenko}, there exists a separately continuous $f:X\times{Y}\rightarrow{\mathbb{R}}$ whose set of points of discontinuity is $U\times{\{y_0\}}$. It follows that $X$ is not weakly Namioka.
\qed
\end{proof}

\begin{theorem}[Main Theorem]
Let $X$ be either a completely regular separable space or a perfectly normal space. Then $X$ is Baire if and only if $X$ is weakly Namioka.
\end{theorem}

\begin{proof}
Follows immediately from Theorems \ref{Brown}, \ref{big}, and \ref{small}
\qed
\end{proof}

\begin{remark} Since metrizable spaces are perfectly normal, weakly Namioka and Baire spaces coincide in the class of metrizable spaces.
\end{remark}

\begin{remark}
Observe that each proof from this section remains valid if the arbitrary metric space in our definition of a weakly Namioka space is taken to be $\mathbb{R}$. Thus, among completely regular separable spaces and among perfectly normal spaces, these definitions give the same class of spaces.
An analogous result for Namioka spaces -- that among Baire spaces, substituting $\mathbb{R}$ for the arbitrary metric space in the definition of Namioka spaces gives the same class of spaces -- is known from \cite{Bouziad}.
\end{remark}

Recall that for topological spaces $X$ and $Y$, a function $f:X\rightarrow{Y}$ is said to be \emph{quasi-continuous at the point $x$} if for every open set $U$ in $X$ containing $x$ and for every open set $V$ in $Y$ containing $f(x)$ there exists an open, non-empty subset $U'$ of $U$ such that $f(U')\subset{V}$. The function is \emph{quasi-continuous} if it is quasi-continuous at all $x\in{X}$. We say that $f:X\times{Y}\rightarrow{Z}$ is \emph{quasi-continuous with respect to $x$} if $f|_{X\times{\{y\}}}$ is quasi-continuous for all $y\in{Y}$ and that $f$ is \emph{quasi-continuous with respect to $y$} if $f|_{\{x\}\times{Y}}$ is quasi-continuous for all $x\in{X}$. For a topological space $X$, second countable space $Y$, and compact metrizable space $Z$, T. Nagamizu shows in \cite{Nagamizu} that if $f:X\times{Y}\rightarrow{Z}$ is continuous with respect to $y$ and $f|_{X\times\{y\}}$ quasi-continuous for each $y$ from a dense set $E$ in $Y$, then there exists a residual subset $A$ of $X$ such that $f$ is (jointly) continuous on $A\times{Y}$. Note that Nagamizu's result resembles Namioka's famous theorem in \cite{Namioka}, but with $Y$ being second countable. Given the relationship established between Namioka and weakly Namioka spaces, it is thus natural to wonder for what class of spaces the assumption of separate continuity in Namioka's theorem can be weakened.

\begin{question}
Let Y be locally compact and $\sigma$-compact, and let Z be an
arbitrary pseudo-metric space. Determine the class $\mathcal{X}$ such that for any
$X$ in $\mathcal{X}$ and any function $f:X\times{Y}\rightarrow{Z}$ which is quasi-continuous with
respect to $x$ and continuous with respect to $y$, there is a dense $G_{\delta}$ subset
$A$ of $X$ such that $f$ is continuous on $A\times{Y}$.
\end{question}

\section{Examples}

Recall that a \emph{network} in a space $X$ is a collection of subsets $\rho$ such that given any open subset $U$ of $X$ and $x\in{U}$, there is a member $P$ of $\rho$ such that $x\in{P}\subset{U}$.

\begin{example}
In \cite{Talagrand}, Remarque $b$, p. 241, M. Talagrand shows that the function $f:[0,1]\times{C_p([0,1],[0,1])\rightarrow{[0,1]}}$ given by $f(x,y)=y(x)$ is separately continuous and discontinuous at \emph{every} point of $X\times{Y}$. It can be shown that $Y=C_p([0,1],[0,1])$, the function space with the topology of pointwise convergence, is completely regular with a countable network, and as such is hereditarily Lindelof and hereditarily separable (see R. Engelking \cite{Engelking}, Exercise 3.4.H, p. 165 and Theorem 2.6.4, p. 107).
\end{example}

\begin{example}
Answering questions due to A. Alexiewicz, W. Orlicz \cite{Alexiewicz} and J.P. Christensen \cite{Christensen} pertaining to the necessity of the compactness assumption on the second factor $Y$ in the theorem for Namioka spaces, J.B. Brown (\cite{Piotrowski}, p. 313) constructs a separately continuous real-valued function defined on the Cartesian product of the closed interval $[0,1]$  and the topological sum of $\mathfrak{c}$ many intervals -- a complete metric space -- such that the conclusion of the theorem for Namioka spaces fails.

Answering Problem C, p. 203 of \cite{Henriksen}, the first-named author \cite{Piotrowski2} refines Brown's techniques (``two-dimensional" example) by constructing a separately continuous real-valued function $f$ defined on the Cartesian product of two complete metric spaces $X,Y$ such that the (in fact, dense $G_\delta$) set $C(f)$ of points of (joint) continuity \emph{fails} to contain either $A\times{Y}$ or $X\times{B}$ for \emph{any} dense $G_\delta$-set $A$ in $X$ or \emph{any} dense $G_\delta$ set $B$ in $Y$. In other words, the condition of the theorem for Namioka spaces fails ``in both directions". 
\end{example}

Recall (\cite{Gillman}, Chapter 4) that a point $x\in{X}$ is called a \emph{P-point} if any $G_\delta$ set containing $x$ is a neighborhood of $x$, and a 
space $X$ is called a \emph{P-space} if each $x\in{X}$ is a P-point.

   The classical example of a discontinuous separately continuous function $sp : {\mathbb{R}\times{\mathbb{R}}}\rightarrow{\mathbb{R}}$ 
is defined by:

  $sp(x,y) = 
\left\{
\begin{array}{lr}
\frac{2xy}{x^2+y^2}&(x,y)\neq{(0,0)}\\
0&(x,y)=(0,0)
\end{array}
\right.$

\noindent One may think that any non-discrete completely regular (Hausdorff) spaces $X$, $Y$ admit a discontinuous, separately continuous
function $f: X\times{Y}\rightarrow{\mathbb{R}}$, but this is not true. In fact:

\begin{proposition}
\label{HWprop} (\cite{Henriksen}, Theorem 6.14, p.196 ) Let $X,Y$ be completely regular spaces, $x_0\in{X}$ be a P-point and $y_0\in{Y}$ have a separable neighborhood. If $f:X\times{Y}\rightarrow{\mathbb{R}}$ is separately continuous, then $f$ is continuous at $(x_0,y_0)$.
\end{proposition}

   Hence one must impose certain restrictions on the nature of non-isolated points in $X$ and $Y$ to guarantee the existence of a separately continuous function. Following \cite{Banakh} one may ask the following natural questions:
 Suppose $X$ and $Y$ are completely regular spaces with non-isolated $G_\delta$ points $x\in{X}$ and $y\in{Y}$. Is there a separately continuous 
function $f : X\times{Y}\rightarrow{\mathbb{R}}$ that is discontinuous at $(x,y)$? Can such a function $f$ be chosen of the form $f = sp\circ{( g\times{h})}$  for suitable 
continuous functions $g : X\rightarrow{\mathbb{R}}$ and $h : Y\rightarrow{\mathbb{R}}$?

  It is shown in \cite{Banakh} that this cannot be done in ZFC. Under Martin's Axiom these questions have negative answers. On the 
other hand there are models of ZFC (e.g., Near Coherence of P-Filters) in which the answers to these questions are affirmative. See \cite{Banakh} for more details.

 We now generalize Proposition \ref{HWprop}:

\begin{theorem}
\label{counter}
Let $x_0\in{X}$ be a P-point, $y_0\in{Y}$ have a separable neighborhood, and $Z$ be regular. If $f:X\times{Y}\rightarrow{Z}$ is separately continuous, then $f$ is continuous at $(x_0,y_0)$.
\end{theorem}

\begin{proof}
Assume $x_0\in{X}$, $y_0\in{Y}$, $Z$, and $f:X\times{Y}\rightarrow{Z}$ are as above. 
Let $S$ be a separable neighborhood of $y_0$,
$D$ a countable dense subset of $S$, and $V$ an open neighborhood of $f(x_0,y_0)$ in $Z$.
By the regularity of $Z$, we can choose an open $V^*$ in $Z$ such that $f(x_0,y_0)\in{V^*}\subset{cl(V^*)}\subset{V}$, where $cl(V^*)$ denotes the closure of $V^*$ in $Z$.
Using continuity with respect to $y$, we may assume without loss of generality that $f(\{x_0\}\times{S})\subset{V^*}$.
By continuity with respect to $x$ we have for any $y\in{D}$ some $U_y$ open in $X$ such that $x_0\in{U_y}$ and $f(U_y\times{\{y\}})\subset{V^*}$.
Since $x_0$ is a P-point, there exists an open set $U\subset{\bigcap_{y\in{D}}U_y}$ containing $x_0$.
As $f(U_y\times{\{y\}})\subset{V^*}$ and $U\subset{U_y}$ for each $y\in{D}$, we have $f(U\times{D})\subset{V^*}$. 
Since $f$ is continuous with respect to $y$, it follows that $f(U\times{S})\subset{f(U\times{cl(D))}}\subset{cl(V^*)}\subset{V}$.
Therefore $f$ is continuous at $(x_0,y_0)$.
\qed
\end{proof}

\begin{remark}
Example 6.16 of \cite{Henriksen} shows that the local separability condition on $Y$ is necessary to Proposition \ref{HWprop} and Theorem \ref{counter}.
\end{remark}

\begin{example}
In Example 3.2 of \cite{Gruenhage}, G. Gruenhage and D. Lutzer construct a Lindelof, hereditarily paracompact, linearly ordered (thus Hausdorff and completely normal) P-space that is not a Baire space. By Theorem \ref{counter}, this space is weakly Namioka.
\end{example}

Following \cite{vanDouwen}, we call a space $X$ \emph{ultradisconnected} if it is crowded (has no isolated points) and if every two disjoint crowded subsets of $X$ have disjoint closures.

\begin{example}
In Example 3.3 of \cite{vanDouwen}, E.K. van Douwen constructs a countable (thus separable), regular (and Hausdorff), ultradisconnected space $X$.
As an ultradisconnected space, $X$ has no isolated points, and so is the countable union of closed nowhere dense sets (its single points) and is therefore not a Baire space.
For any second countable space $Y$, the product $X\times{Y}$ is first countable. Ultradisconnected spaces are extremally disconnected \cite{vanDouwen}, so any sequence $(x,y)_n$ in $X\times{Y}$ converging to $(x,y)$ must be eventually constant in $X$. Thus for any metric space Z and any function $f:X\times{Y}\rightarrow{Z}$ continuous with respect to $y$, the image of $(x,y)_n$ under $f$ must converge to $f(x,y)$. It follows that any separately continuous function $X\times{Y}\rightarrow{Z}$ is continuous and so $X$ is weakly Namioka.
\end{example}

\section*{Acknowledgements}
The authors would like to thank J.B. Brown for contributing the proof of Lemma \ref{BrownLemma}.












\bibliographystyle{plain}

\end{document}